\documentclass[12pt,draftcls,onecolumn]{IEEEtran}
\usepackage{amsmath}
\usepackage{graphicx}
\usepackage{stmaryrd}
\usepackage{amsfonts}
\usepackage{graphicx,times,amsmath}
\usepackage{graphics,color}
\usepackage{graphicx}
\usepackage{latexsym}
\usepackage{amsfonts}
\usepackage{amssymb}
\usepackage{url}
\usepackage{subfigure}
\usepackage{times}
\usepackage{color}
\usepackage{cite}
\newtheorem{theorem}{Theorem}
\newtheorem{lemma}{Lemma}

\newtheorem{corollary}{Corollary}
\newtheorem{definition}{Definition}
\newtheorem{remark}{Remark}
\newtheorem{assumption}{Assumption}

\DeclareMathOperator{\diag}{diag}
\DeclareMathOperator{\Prb}{\mathbb{P}}
\DeclareMathOperator{\E}{\mathbb{E}}
\DeclareMathOperator{\Sp}{Sp}
\DeclareMathOperator{\ST}{ST}
\begin{document}
%
\title{Structure-Based Self-Triggered Consensus in Networks of Multiagents with Switching Topologies}

\author{Bo Liu~\IEEEmembership{Member, IEEE}, Wenlian Lu~\IEEEmembership{Member, IEEE}, Licheng Jiao ~\IEEEmembership{Senior Member, IEEE}, and Tianping Chen
        ~\IEEEmembership{Senior Member, IEEE}
\thanks{This work is jointly supported by the National Basic Research Program (973 Program) of China (No. 2013CB329402), the National Natural Science Foundation of China(No.61473215), The Fund for Foreign Scholars in University Research and Teaching Programs (the 111 Project) (No. B07048), National Science Foundation of China under Grant No. 91438103, 91438201, the Program for Cheung Kong Scholars and Innovative Research Team in University (No. IRT1170), the Fundamental Research Funds for the Central Universities under Grant No. XJS14035. Marie Curie International Incoming Fellowship from the European
Commission (FP7-PEOPLE-2011-IIF-302421), the National Natural Sciences Foundation
of China (Nos. 61273211 and 61273309), and the Program for New Century Excellent Talents in University (NCET-13-0139)
}
\thanks{B. Liu and L. Jiao are with Key Laboratory of Intelligent Perception and Image Understanding of Ministry of Education,
International Research Center for Intelligent Perception and Computation,
Xidian University, Xi'an, Shaanxi Province 710071, China
(email: liub@xidian.edu.cn, lchjiao@mail.xidian.edu.cn); W. Lu is with the Centre for Computational Systems Biology and School of Mathematical Sciences, Fudan University, Shanghai, 200433, China(email: wenlian@fudan.edu.cn); T. Chen is with the School of Computer Science and School of Mathematical Sciences, Fudan University, Shanghai 200433, China(Email: tchen@fudan.edu.cn).}
\thanks{}}
\maketitle

\begin{abstract}
In this paper, we propose a new self-triggered consensus algorithm in networks of multi-agents. Different from existing works, which are based on the observation of states, here, each agent determines its next update time based on its coupling structure. Both centralized and distributed approaches of the algorithms have been discussed. By transforming the algorithm to a proper discrete-time systems without self delays, we established a new analysis framework to prove the convergence of the algorithm. Then we extended the algorithm to networks with switching topologies, especially stochastically switching topologies. Compared to existing works, our algorithm is easier to understand and implement. It explicitly provides positive lower and upper bounds for the update time interval of each agent based on its coupling structure, which can also be independently adjusted by each agent according to its own situation. Our work reveals that the event/self triggered algorithms are essentially discrete and more suitable to a discrete analysis framework. Numerical simulations are also provided to illustrate the theoretical results.

\end{abstract}

\pagestyle{plain}

\section{Introduction}\quad

Distributed control of networked cooperative multiagent systems has received much attention in recent years due to the rapid advances in computing and communication technologies. Examples include agreement or consensus problems\cite{RenAtkins_IJRNC_2007,JiEgerstedt_TR_2007,Olfati-SaberMurray_TAC_2004}, in which a group of agents seek to agree upon certain quantities of interest, formation control of robots and vehicles\cite{ConsoliniMorbidiPrattichizzoTosques_Automatica_2008,CaoAndersonMorseYu_CDC_2008}, and distributed estimation\cite{SperanzonFischioneJohansson_CDC_2006,Olfati-SaberShamma_CDC_2005}, etc.

In these applications, an important aspect is to determine the controller actuation schemes. Although continuous feedback of states have always been used in early implementations of distributed control, it is not suitable for agents equipped with embedded microprocessors that have very limited resources to transmit and gather information. To overcome this difficulty, event-triggered control scheme\cite{FanFengWangSong_Automatica_2013,SeybothDimarogonasJohansson_Automatica_2013,LiuChenYuan_JSSC_2012,ManuelTabuada_TAC_2011,WangLemmon_TAC_2011,WangLemmon_CDC_2008,HeemelsSandeeBosch_IJC_2007,Dimarogonas_TAC_2012,Tabuada_TAC_2007} was proposed to reduce the controller updates. In fact, event-driven strategies
for multi-agent systems can be viewed as a linearization and discretization
process, which was considered and investigated in early papers
\cite{LC2004,LC2007}. For example, in the paper \cite{LC2004}, following
algorithm was investigated
\begin{align}
x^{i}(t+1)=f(x^{i}(t))+c_{i}\sum_{j=1}^{m}a_{ij}(f(x^{j}(t)))\label{cml}
\end{align}
which can be considered as nonlinear consensus algorithm. In particular, let $f(x(t))=x(t)$ and $c_{i}=(t_{k+1}^{i}-t_{k}^{i})$.
Then (\ref{cml}) becomes
\begin{align}
x^{i}(t_{k+1}^{i})=x^{i}(t_{k}^{i})+(t_{k+1}^{i}-t_{k}^{i})\sum_{j=1}^{m}a_{ij}x^{j}(t_{k}^{i}),
\end{align}
which is some variant of the event triggering (distributed, self triggered) model for
consensus problem. In centralized control, the bound for
$(t_{k+1}^{i}-t_{k}^{i})=(t_{k+1}-t_{k})$ to reach synchronization was
given in the paper \cite{LC2004} when the coupling graph is indirected and
in \cite{LC2007} for the directed coupling graph.
The key point in event-triggered algorithm is the design of a decision maker that determines when the next actuator update should occur. The existing event-triggered algorithms are all based on the observation of states. For example, in a typical event-triggered algorithm proposed in \cite{Tabuada_TAC_2007}, an update is triggered when a certain error of the states becomes large enough with respect to the norm of the state, which requires a continuous observation of the states. In addition, self-triggered control\cite{AntaTabuada_TAC_2010,ManuelAntaTabuada_Automatica_2010,WangLemmon_TAC_2009,MazoTabuada_CDC_2008,AntaTabuada_PACC_2008} has been proposed as a natural extension of event-triggered control, in which each agent predicts its next update time based on discontinuous state observation to further reduce resource usage for the control systems.

Both the event-triggered and self-triggered control algorithms that have been proposed till now have some drawbacks in common. First, the resulting system in event-triggered control is one that with system delays, especially with self delays, which generally is difficult to handle. And the existing analysis for such algorithms are always based on some quadratic Lyapunov function, which has very strict restrictions on the network structure. For example, to the best of our knowledge, the latest analysis is still restricted to static networks. Second, since these algorithms are based on the observation of states, which are generally much complicated and untraceable, making it difficult to predict and exclude some unexpected possibilities such as the occurrence of zero inter-execution times.

To overcome such difficulties, we proposed a new kind of self-triggered consensus algorithms that are structure-based. That is, each agent predicts its next update time based on its coupling structure with its neighbours instead of the observation of their states. This can bring several advantages. First, since the coupling structure is relatively simpler and more traceable than the states, it is relatively easier to handle. Actually, in our algorithm, the occurrence of zero inter-execution time has been excluded by directly providing a positive lower bound on the update intervals.

Second, although the resulting system of the self-triggered algorithm is one with self delays, by some proper transformation, we showed that the system can corresponding to a discrete-time system without self-delays. We have proposed the algorithm in both centralized and distributed approaches. In the centralized approach, the system can be directly related to a discrete-time system without delays, which can be seen as a discrete version of the nominal system. However, in the distributed approach, the situation is much more complicated and the system can not be directly transformed into its discretized version. However, by some proper indirect transformation, we showed that the convergence of the nominal system can be reduced to that of a discrete-time consensus algorithm with off-diagonal delays but without self delays. Thus, other than using a quadratic Lyapunov function, the convergence of the original algorithm can be solved by the analysis of a discrete-time consensus algorithm. This brings another possibility that is also considered as a big advantage and a major contribution of this work. That is, we can analyze networks with switching topologies, especially stochastically switching topologies.

The rest of the paper is organized as follows: Section \ref{secPreliminaries} provides some preliminaries on matrix and graph theory that will be used in the main text. Section \ref{secSelfTriggered} provides the self-triggered consensus algorithm in both centralized and distributed approach with convergence analysis. Section \ref{secSimulation} provide an example with numerical simulation to illustrate the theoretical results obtained in the previous section. The paper is concluded in Section \ref{secConclusion}.

\section{Preliminaries}\label{secPreliminaries}

In this section, we present some definitions and results in matrix and graph theories and probability theories that will be used later.

For a matrix $A=[a_{ij}]$, $a_{ij}$ represents the entry of $A$ on the $i$th row and $j$th column, which is sometimes also denoted as $[A]_{ij}$. A matrix $A=[a_{ij}]$ is called a {\em nonnegative matrix} if $a_{ij}\ge 0$ for all $i$, $j$. And $A$ is called a {\em stochastic matrix} if $A$ is square and $\sum_{j}a_{ij}=1$ for each $i$. Given a nonnegative matrix $A$ and $\delta>0$, the {\em $\delta$-matrix} of $A$ is a matrix that has nonzero entries only in the place that $A$ has entries equal to or greater than $\delta$. For two matrix $A$, $B$ of the same dimension, we write $A\ge B$ if $A-B$ is a nonnegative matrix. Throughout this paper, we use $\prod_{i=1}^{k}A_{i}=A_{k}A_{k-1}\cdots A_{1}$ to denote the left product of matrices.

A {\em directed graph} $\mathcal{G}$ is defined by its vertex set $\mathcal{V}(\mathcal{G})=\{v_{1},\cdots,v_{n}\}$ and edge set $\mathcal{E}(\mathcal{G})\subseteq \{(v_{i},v_{j}): v_{i}, v_{j}\in \mathcal{V}(\mathcal{G})\}$, where an edge is an ordered pair of vertices in $\mathcal{V}(\mathcal{G})$. If $(v_{i},v_{j})\in \mathcal{E}(\mathcal{G})$, then $v_{i}$ is called the {\em (in-)neighbour} of $v_{j}$. $\mathcal{N}_{i}=\{j:~(v_{j},v_{i})\in \mathcal{E}(\mathcal{G})\}$ is the set of neighbours of agent $i$. A {\em directed path} in $\mathcal{G}$ is an ordered sequence of vertices $v_{1}$, $\cdots$, $v_{s}$ such that $(v_{i},v_{i+1})\in \mathcal{E}(\mathcal{G})$ for $i=1,\cdots,s-1$. A {\em directed tree} is a directed graph where every vertex has exactly one in-neighbour except one ,called the {\em root}, without any in-neighbour. And there exists a directed path from the root to any other vertex of the graph. A subgraph of $\mathcal{G}$ is a directed graph $\mathcal{G}_{S}$ satisfying $\mathcal{V}(\mathcal{G}_{S})\subseteq \mathcal{V}(\mathcal{G})$, and $\mathcal{E}(\mathcal{G})\subseteq \mathcal{E}(\mathcal{G})$. A {\em spanning subgraph} of $\mathcal{G}$ is a subgraph of $\mathcal{G}$ that has the same vertex set with $\mathcal{G}$. We say $\mathcal{G}$ has a {\em spanning tree} if $\mathcal{G}$ has a spanning subgraph that is a tree.

A {\em weighted directed graph} is a directed graph where each edge is equipped with a weight. Thus, a graph $\mathcal{G}$ of $n$ vertices corresponds to an $n\times n$ nonnegative matrix $A=[a_{ij}]$, called the {\em weight matrix} in such way that $(v_{j},v_{i})\in \mathcal{E}(\mathcal{G})$ if and only if $a_{ij}>0$. On the other hand, given an $n\times n$ nonnegative matrix $A$, it corresponds to a weighted directed graph $\mathcal{G}(A)$ such that $\mathcal{G}(A)$ has $A$ as its weight matrix. Using this correspondence, we can introduce the concept of $\delta$-graph. The {\em $\delta$-graph} of $\mathcal{G}$ is a weighted directed graph that has the $\delta$-matrix of $\mathcal{G}$'s weight matrix as its weight matrix. A graph $\mathcal{G}$ has a {\em $\delta$ spanning tree} if its $\delta$-graph has a spanning tree.

From the weight matrix $A=[a_{ij}]$ of a graph $\mathcal{G}$, we define its graph Laplacian $L=[l_{ij}]$ as follows:
\begin{eqnarray*}
  l_{ij}=\left\{
  \begin{array}{cc}
    \sum_{k\ne i}a_{ik}, & i=j;\\
    -a_{ij}, & i\ne j.
  \end{array}
  \right.
\end{eqnarray*}
Thus $L$ has zero row sum with $l_{ii}>0$ while $l_{ij}\le 0$ for $i,j=1,2,\cdots,n$ and $i\ne j$. And the nonnegative linear combination of several graph Laplacians is also a graph Laplacian of some graph. There is a one-to-one correspondence between a graph and its weight matrix or its Laplacian matrix. In the following, for the sake of simplicity in presentation, sometimes we don't explicitly distinguish a graph from its weight matrix or Laplacian matrix, i.e., when we say a matrix has some property that is usually associated with a graph, what we mean is that property is held by the graph corresponding to this matrix. For example, when we say a nonnegative matrix $A$ has a spanning tree, what we mean is that the graph of $A$ has a spanning tree. And it is of similar meaning when we say that a graph has some property that is usually associated with a matrix.

Let $\{\Omega, \mathcal{F},\Prb\}$ be a probability space, where $\Omega$ is the sample space, $\mathcal{F}$ is $\sigma$-algebra on $\Omega$, and $\Prb$ is the probability on $\mathcal{F}$. We use $\E\{\cdot\}$ to denote the mathematical expectation and $\E\{\cdot|\mathcal{F}\}$ the conditional expectation with respect to $\mathcal{F}$, i.e., $\E\{\cdot|\mathcal{F}\}$ is a random variable that is measurable with respect to $\mathcal{F}$.

\begin{definition}[adapted process/sequence]
Let $\{A_{k}\}$ be a stochastic process defined on the basic
probability space $\{\Omega,\mathcal{F},\Prb\}$, and let
$\{\mathcal{F}_{k}\}$ be a filtration, i.e., a sequence of nondecreasing
sub-$\sigma$-algebras of $\mathcal{F}$. If $A_{k}$ is measurable with respect to
$\mathcal{F}_{k}$, then the sequence
$\{A_{k},\mathcal{F}_{k}\}$ is called an adapted process/sequence.
\end{definition}
\begin{remark}
For example, let $X(\omega)$ be a random variable defined on $\Omega$, then $\{\E\{X(\omega)|\mathcal{F}_{k}\},\mathcal{F}_{k}\}$ is an adapted sequence.
\end{remark}

\section{Self-Triggered Consensus Algorithm}\label{secSelfTriggered}
In a network of $n$ agents, with $x_{i}\in \mathbb{R}$ being the state of agent $i$,
we assume that each agent's dynamics obey a single integrator model
\begin{eqnarray}
  \dot{x}_{i}(t)=u_{i}(t), ~~i=1,2,\cdots,n,
\end{eqnarray}
where $u_{i}(t)$ is the control law for agent $i$ at time $t$. In consensus algorithm, the control law is usually given by(\cite{Olfati-SaberMurray_TAC_2004,FaxMurray_IFAC_2002,DimarogonasKyriakopoulos_Automatica_2008})
\begin{eqnarray}
  u_{i}(t)=-\sum_{j\in \mathcal{N}_{i}}l_{ij}[x_{j}(t)-x_{i}(t)]
  =-\sum_{j=1}^{n}l_{ij}[x_{j}(t)-x_{i}(t)].
\end{eqnarray}
Thus the consensus algorithm can be written as
\begin{eqnarray}
  \dot{x}_{i}(t)=-\sum_{j=1}^{n}l_{ij}[x_{j}(t)-x_{i}(t)],
\end{eqnarray}
or in matrix form as
\begin{eqnarray}
  \dot{x}(t)=-Lx(t),
\end{eqnarray}
where $x(t)=[x_{1}(t),\cdots, x_{n}(t)]^{\top}\in \mathbb{R}^{n}$ is the stack vector of the agents' states. In event/self-triggered control, the control law $u(t)$ is piecewise constant, i.e.,
\begin{eqnarray}
  u(t)=u(t_{k}), ~t\in [t_{k},t_{k+1}),
\end{eqnarray}
where $\{t_{k}\}$ is the time sequence that an update of $u(\cdot)$ occurs. The central point of such algorithm is the choice of appropriate $\{t_{k}\}$ such that some desired properties of the algorithm such as stability and convergence can be preserved. This can be done in a centralized approach or a distributed approach, both of which will be discussed in the following.

\subsection{Centralized Approach}

In this subsection, we first consider centralized self-triggered consensus algorithm. The sequence of the update time is denoted by: $t_{0}, t_{1}, t_{2}, \cdots$, then the self-triggered algorithm has the form:
\begin{eqnarray}\label{eqnCentralizedFix}
  \dot{x}(t)=-Lx(t_{k}),~t\in [t_{k},t_{k+1}).
\end{eqnarray}

Denote $\Delta t_{k}=t_{k+1}-t_{k}$ for $k=0,1,\cdots$. We have the following Theorem.

\begin{theorem}\label{thmCentralFixed}
If the graph $\mathcal{G}(L)$ has a spanning tree and there exists $\delta\in (0,1/2)$ such that $\delta/l_{\max}\le\Delta t_{k}\le (1-\delta)/l_{\max}$ for each $k$, where $l_{\max}=\max_{i}\{l_{ii}\}$, then the algorithm will achieve consensus asymptotically.
\end{theorem}

\begin{proof}
From the algorithm, for each $t\in (t_{k},t_{k+1}]$, since
$$\dot{x}(t)=-Lx(t_{k}),$$ we have
$$x(t)=x(t_{k})-(t-t_{k})Lx(t_{k})=[I-(t-t_{k})L]x(t_{k}).$$
Particularly, let $t=t_{k+1}$, we have
$$x(t_{k+1})=(I-\Delta t_{k}L)x(t_{k})= A_{k}x(t_{k}),$$
where $A_{k}=I-\Delta t_{k}L$. It is easy to verify that $A_{k}$ is a stochastic matrix for each $k$. In fact, since $\delta/l_{\max}\le\Delta t_{k}\le (1-\delta)/l_{\max}$,
$$[A_{k}]_{ii}=1-\Delta t_{k}l_{ii}\ge 1-\Delta t_{k}l_{\max}\ge 1-\frac{1-\delta}{l_{\max}}l_{\max}=\delta.$$ For $i\ne j$,
$$[A_{k}]_{ij}=-\Delta t_{k} l_{ij}\ge 0.$$ It implies that $[A_{k}]_{ij}$ is positive if and only if $-l_{ij}$ is positive, and the positive elements of $[A_{k}]$ is uniformly lower bounded by a positive scalar since both $\Delta t_{k}$ and $-l_{ij}$ is uniformly lower bounded by a positive scalar. Furthermore,
$\sum_{j=1}^{n}[A_{k}]_{ij}=1-\Delta t_{k}\sum_{j=1}^{n}l_{ij}=1.$
By a standard argument from the theory of products of stochastic matrices, there exists $x^{*}\in \mathbb{R}$ such that
$$\lim_{k\to+\infty}x_{i}(t_{k})=x^{*}.$$
On the other hand, it is not difficult to verify that the function $\max_{i}\{x_{i}(t)\}$ is nonincreasing and $\min_{i}\{x_{i}(t)\}$ is nondecreasing. Thus,
$$\lim_{t\to+\infty}x_{i}(t)=x^{*}.$$
\end{proof}

Using a similar argument as that in Theorem \ref{thmCentralFixed}, it is not difficult to propose a centralized event-triggered algorithm on networks with stochastic switching topologies. For example, consider the following consensus algorithm:
\begin{eqnarray}\label{eqnCentralSwitching}
  \dot{x}(t)=-L^{k}x(t_{k}),~t\in [t_{k},t_{k+1}),
\end{eqnarray}
where $L^{k}=[l_{ij}^{k}]$ is uniformly bounded in $k$ and $l_{\max}^{k}=\max_{i}\{l_{ii}^{k}\}\in [l_{\min},l_{\max}]$ for some $0<l_{\min}<l_{\max}$. Before provide the next theorem, we first summarize the following assumption.

\begin{assumption}\label{assumCentralSwitching}
\begin{enumerate}
  \item[(i)] There exists $\delta\in (0,1/2)$ such that $\Delta t_{k}\in [\delta/l_{\max}^{k},(1-\delta)/l_{\max}^{k}]$;
  \item[(ii)]$\{L^{k},\mathcal{F}_{k}\}$ is an adapted random sequence, and there exists $h>0$, $\delta'>0$ such that the graph corresponding to the conditional expectation $\E\{\sum_{m=k+1}^{k+h}L^{m}|\mathcal{F}_{k}\}$ has a $\delta'$-spanning tree;
\end{enumerate}
\end{assumption}
Now we have
\begin{theorem}
Under assumption\ref{assumCentralSwitching}, the self-triggered consensus algorithm \eqref{eqnCentralSwitching} will reach consensus almost surely.
\end{theorem}
\begin{proof}
Similar as in the proof of Theorem \ref{thmCentralFixed}, we have $x(t_{k+1})=A_{k}x(t_{k})$ with $A_{k}=I-\Delta t_{k}L^{k}$. Since $\{\Delta t_{k}\}$ is a deterministic sequence, $\{A_{k},\mathcal{F}_{k}\}$ is also an adapted sequence, and
\begin{eqnarray*}
\E\{\sum_{m=k+1}^{k+h}A_{m}|\mathcal{F}_{k}\}&=&hI-\E\{\sum_{m=k+1}^{k+h}\Delta t_{m}L^{m}|\mathcal{F}_{k}\}\\
&\ge&h\delta I-\frac{\delta}{l_{\max}}\E\{\sum_{m=k+1}^{k+h}L_{0}^{m}|\mathcal{F}_{k}\},
\end{eqnarray*}
where $[L^{m}_{0}]_{ij}=[L^{m}]_{ij}$ for $i\ne j$ and $[L_{0}^{m}]_{ii}=0$ for all $i$. Thus, $\E\{\sum_{m=k+1}^{k+h}A_{m}|\mathcal{F}_{k}\}$ has a $\delta''$-spanning tree with $\delta''=\delta'\delta/l_{\max}>0$. From Theorem 3.1 of \cite{LiuB_SICON_2011}, the sequence $\{x(t_{k})\}$ will reach consensus almost surely, thus $x(t)$ will also reach consensus almost surely.
\end{proof}
As corollaries in special cases, it can be shown that almost sure consensus still holds if we replace item (ii) in Assumption \ref{assumCentralSwitching} with one of the following more special ones.
\begin{enumerate}
  \item[(ii')]$\{L^{k}\}$ is an independent and identically distributed sequence, and there exists $\delta'>0$ such that the graph corresponding to the expectation $\E L^{k}$ has a $\delta'$-spanning tree;
  \item[(ii'')]$\{L^{k}\}$ is a homogenous Markov chain with a stationary distribution $\pi$, and there exists $\delta'>0$ such that the graph corresponding to the expectation with respect to $\pi$, $\E_{\pi}L^{k}$, has a $\delta'$-spanning tree;
\end{enumerate}

\subsection{Distributed Approach}

In this section, we discuss the distributed event-triggered consensus algorithm.

Let $\{t_{i}^{k}\}_{k=0}^{+\infty}$ be the time sequence such that $t_{i}^{k}$ is the $k$th time that agent $i$ updates its control law. Then the distributed self-triggered algorithm has the following form:
\begin{eqnarray}
  x_{i}(t)=-\sum_{j=1}^{n}l_{ij}x_{j}(t_{j}^{k_{j}(t)}),~~i=1,\cdots,n,
\end{eqnarray}
where $k_{j}(t)=\max\{k:~t_{j}^{k}\le t\}$.

Denote $\Delta t_{i}^{k}=t_{i}^{k+1}-t_{i}^{k}$. Then we have the following Theorem.

\begin{theorem}\label{thmDistributedFixTopology}
If the graph $\mathcal{G}(L)$ has a spanning tree and there exists $\delta_{i}\in (0,1/2)$ for $i=1,2,\cdots,n$ such that $\Delta t_{i}^{k}\in [\delta_{i}/l_{ii},(1-\delta_{i})/l_{ii}]$, then the distributed event triggered algorithm can reach a consensus as $t\to \infty$.
\end{theorem}

The proof will be divided into two steps corresponding to two lemmas. Now we give and prove the first lemma.
\begin{lemma}
If there exists $x^{*}\in \mathbb{R}$ such that
$$
\lim_{k\to \infty}x(t_{i}^{k})=x^{*}, ~~i=1,2,\cdots,n,
$$
then the distributed event-triggered algorithm will reach a consensus, i.e.,
$$
\lim_{t\to\infty}x_{i}(t)=x^{*}, ~~i=1,2,\cdots,n.
$$
\end{lemma}
\begin{proof}
By the definition of the algorithm, for agent $i$,
$$
\dot{x}_{i}(t)=-\sum_{j=1}^{n}l_{ij}x(t_{j}^{k_{j}(t)})=-\sum_{j=1,j\ne i}^{n}l_{ij}[x_{j}(t_{j}^{k_{j}(t)})-x_{i}(t_{i}^{k_{i}(t)})].
$$
Since $\lim_{t\to\infty}k_{j}(t)=+\infty,$ $$
\lim_{t\to\infty}\dot{x}_{i}(t)=-\sum_{j=1,j\ne i}^{n}l_{ij}(x^{*}-x^{*})=0.
$$
Thus,
\begin{eqnarray*}
\lim_{t\to\infty}|x_{i}(t)-x^{*}|&\le & \lim_{t\to\infty}|x_{i}(t)-x_{i}(t_{i}^{k_{i}(t)})|
+\lim_{t\to\infty}|x_{i}(t_{i}^{k_{i}(t)})-x^{*}|\\
&\le&\lim_{t\to\infty}\Delta {t_{i}^{k_{i}(t)}}\max_{s\in [t_{i}^{k_{i}(t)},t]}|\dot{x}_{i}(s)|+\lim_{t\to\infty}|x_{i}(t_{i}^{k_{i}(t)})-x^{*}|\\
&\le& \frac{1-\delta_{i}}{l_{ii}}\lim_{t\to\infty}\max_{s\in [t_{i}^{k_{i}(t)},t]}|\dot{x}_{i}(s)|+\lim_{t\to\infty}|x_{i}(t_{i}^{k_{i}(t)})-x^{*}|\\
&=&0.
\end{eqnarray*}

\end{proof}

Now we are to provide and prove the second lemma.

\begin{lemma}
Under the assumption in Theorem \ref{thmDistributedFixTopology}, there exists $x^{*}$ such that
$$
\lim_{k\to\infty}x_{i}(t_{i}^{k})=x^{*}.
$$
\end{lemma}

\begin{proof}
Let $\{t_{k}\}$ be the time sequence that the $k$th update occurred in the network, i.e., $\{t_{k}\}=\cup_{i=1}^{n}\{t_{i}^{k'}\}$, and $\Delta t_{k}=t_{k+1}-t_{k}$. In case for that $t_{k}=t_{i}^{k'}=t_{j}^{k''}$ for some some $k$, $k'$, $k''$ and $i\ne j$, we consider the update of agent $i$ and $j$ at time $t_{k}$ as one update of the network. Now, construct an auxiliary sequence $\{y(k)\}$, where $y(k)=[y_{1}(k),y_{2}(k),\cdots,y_{n}(k)]^{\top}\in \mathbb{R}^{n}$ with $y_{i}(k)=x_{i}(t_{i}^{k_{i}(t_{k})})$, i.e.,
 $y_{i}(k)$ is the latest state that agent $i$ uses in its control law at time $t_{k}$. Now consider the evolution of $\{y(k)\}$.

\begin{enumerate}
  \item[Case 1.] Agent $i$ does not update at time $t_{k+1}$:
  In this case, by definition,
  $$y_{i}(k+1)=y_{i}(k);$$

  \item[Case 2.] Agent $i$ does update at time $t_{k+1}$:
  In this case, $t_{i}^{k_{i}(t_{k+1})}=t_{k+1}$, assume $t_{i}^{k_{i}(t_{k})}=t_{k-d_{ik}}$ be the last update of agent $i$ before $t_{k+1}$, with $d_{ik}\ge 0$ being the number of updates occur at all other agents $x_{j}$, $j\ne i$, between the two successive updates of agent $i$, i.e., in the interval $(t_{i}^{k_{i}(t_{k})},t_{i}^{k_{i}(t_{k+1})})$. By definition $y_{i}(k)=\cdots=y_{i}(k-d_{ik})$, and
  $\dot{x}_{i}(t)=-\sum_{j=1}^{n}l_{ij}y_{j}(k')$ for $t\in (t_{k'},t_{k'+1})$ for all $k'$. Since $[t_{i}^{k_{i}(t_{k})},t_{i}^{k_{i}(t_{k+1})})=\cup_{m=k-d_{ik}}^{k}[t_{m},t_{m+1})$,

  \begin{eqnarray*}
  y_{i}(k+1)&=&x_{i}(t_{k+1})=x_{i}(t_{i}^{k_{i}(t_{k})})+\int_{t_{i}^{k_{i}(t_{k})}}^{t_{i}^{k_{i}(t_{k+1})}}\dot{x}_{i}(t)dt\\
  &=&x_{i}(t_{k-d_{ik}})+\sum_{m=0}^{d_{ik}}\int_{t_{k-d_{ik}+m}}^{k-d_{ik}+m+1}\dot{x}_{i}(t)dt\\
  &=&x_{i}(t_{k-d_{ik}})-\sum_{m=0}^{d_{ik}}\Delta t_{k-d_{ik}+m}\sum_{j=1}^{n}l_{ij}y_{j}(k-d_{ik}+m)\\
  &=&y_{i}(k-d_{ik})-\sum_{m=0}^{d_{ik}}\Delta t_{k-d_{ik}+m}l_{ii}y_{i}(k-d_{ik}+m)\\
  &&-\sum_{m=0}^{d_{ik}}\Delta t_{k-d_{ik}+m}\sum_{j=1,j\ne i}^{n}l_{ij}y_{j}(k-d_{ik}+m)\\
  &=&y_{i}(k)-\sum_{m=0}^{d_{ik}}\Delta t_{k-d_{ik}+m}l_{ii}y_{i}(k)-\sum_{m=0}^{d_{ik}}\Delta t_{k-d_{ik}+m}\sum_{j=1,j\ne i}^{n}l_{ij}y_{j}(k-d_{ik}+m)\\
  &=&(1-l_{ii}\sum_{m=0}^{d_{ik}}\Delta t_{k-d_{ik}+m})y_{i}(k)-\sum_{m=0}^{d_{ik}}\sum_{j=1,j\ne i}^{n}\Delta t_{k-d_{ik}+m}l_{ij}y_{j}(k-d_{ik}+m)\\
  &=&(1-\Delta t_{i}^{k_{i}(t_{k})}l_{ii})y_{i}(k)-\sum_{m=0}^{d_{ik}}\sum_{j=1,j\ne i}^{n}\Delta t_{k-d_{ik}+m}l_{ij}y_{j}(k-d_{ik}+m)\\
  &=&\sum_{m=0}^{d_{ik}}\sum_{j=1}^{n}a_{ij}^{m}(k)y_{j}(k-m),
  \end{eqnarray*}
  where $a_{ii}^{0}(k)=1-\Delta t_{i}^{k_{i}(t_{k})}l_{ii}$, $a_{ii}^{m}(k)=0$ for $m=1,2,\cdots,d_{ik}$ and $a_{ij}^{m}(k)=-\Delta t_{k-m}l_{ij}$.
  It is easy to verify that
  $a_{ii}^{0}(k)=1-\Delta t_{i}^{k_{i}(t_{k})}l_{ii}\ge 1-(1-\delta_{i})=\delta_{i}$, $a_{ij}^{m}=-\Delta t_{k-m}l_{ij}\ge 0$, and
  $$
  \sum_{m=0}^{d_{ik}}\sum_{j=1}^{n}a_{ij}^{m}(k)=1.
  $$
  Besides, by the assumption $\delta_{i}/l_{ii}\le\Delta t_{i}^{k}\le (1-\delta_{i})/l_{ii}$ for each $i$, $k$, it is clear that  in the interval $[t_{i}^{k},t_{i}^{k+1}]$, only finite updates occur for agents $x_{j}$ with $j\ne i$, and the number of updates is uniformly upper bounded, i.e., there exists $\tau>0$  independent of $i$, $k$, such that $d_{ik}\le \tau$. For example, let $\delta_{\min}=\min_{i}\{\delta_{i}/l_{ii}\}$, and $\delta_{\max}=\max_{i}\{\delta_{i}/l_{ii}\}$, then $\delta_{\min}\le\Delta t_{i}^{k}\le \delta_{\max}$ for all $i$, $k$, and on each time interval of length $m\delta_{\min}$, at most $m+1$ updates occur for each agent. Let $h'$ be the smallest positive integer satisfying $h'\delta_{\min}\ge \delta_{\max}$, then on each time interval of length $\delta_{\max}$, at most $h'+1$ updates occur for each agent. Thus on the interval $(t_{i}^{k},t_{i}^{k+1})$, at most $(n-1)(h'+1)$ updates occur for all agents $j$ with $i\ne j$. Pick $\tau=(n-1)(h'+1)$, then $d_{ik}\le \tau$ for all $i$, $k$.

\end{enumerate}

  If agent $i$ does not update at time $t_{k+1}$, i.e., case 1, we define  $a_{ij}^{m}(k)=0$ for all $m=0,1,\cdots,\tau$ and $i,j=1,2,\cdots,n$ except for $a_{ii}^{0}(k)=1$,

 If agent $i$ updates at time $t_{k+1}$, i.e., case 2, we define $a_{ij}^{m}(k)=0$ for all $i,j=1,2,\cdots,n$ and $m=d_{ik}+1,\cdots,\tau$ when $d_{ik}<\tau$.

 For both cases, we can give a uniform iterative formula for $y(k)$ as:
\begin{eqnarray}
  y(k+1)=\sum_{l=0}^{\tau}\sum_{j=1}^{n}a_{ij}^{l}(k)y_{j}(k-l),
\end{eqnarray}
where
$a_{ij}^{l}(k)\ge 0$ for all $i,j,l,k$, and
$$
\sum_{l=0}^{\tau}\sum_{j=1}^{n}a_{ij}^{l}(k)=1,
$$
for all $i$, $k$.
 Denote $\delta=\min_{i}\{\delta_{i}\}>0$, we have $a_{ii}^{0}(k)\ge \delta$ for all $i$, $k$.

 Construct a new matrix $B(k)=[b_{ij}(k)]$, where
 \begin{align}
 b_{ii}(k)=a_{ii}^{0}(k)
 \end{align}
 and
 \begin{align}
 b_{ij}(k)=\sum_{l=0}^{\tau}a_{ij}^{l}(k),~~ if~ i\ne j
 \end{align}

   Obviously, $B(k)\ge \delta I$, and $b_{ij}(k)=\Delta t_{i}^{k_{i}(t_{k})} l_{ij}\ge \delta_{\min}l_{ij}$ if agent $i$ updates at time $t_{k+1}$. If there exists $h>0$ such that on each interval $[t_{k},t_{k+h}]$ all agents update at least once, then $\sum_{m=k}^{k+h}B(m)\ge -\delta_{\min}L$ for each $k$. Since $\mathcal{G}(L)$ has a spanning tree, there exists $\delta'>0$ such that $\sum_{m=k}^{k+h}B(m)$ has a $\delta'$-spanning tree for each $k$. From Corollary \ref{corollaryAppendixDeterministic} in the appendix, the sequence $\{y(k)\}$ will reach a consensus, i.e., there exists $x^{*}\in \mathbb{R}$ such that
$$
\lim_{k\to\infty}y_{i}(k)=x^{*}.
$$
This is equivalent to
$$
\lim_{k\to\infty}x_{i}(t_{i}^{k})=x^{*}.
$$
At last, we give some hint on how to calculate the quantity $h$ mentioned above. It is easy to see that each agent updates at least once on each time interval of length $\delta_{\max}$, since $\Delta t_{i}^{k}\le \delta_{\max}$ for all $i$, $k$. And in each time interval of length $\delta_{\max}$, there are at most $n(h'+1)$ updates of all the agents, where $h'$ is defined as before. Pick $h=n(h'+1)$, then for each $k$, each agent updates at least once in the time interval $[t_{k},t_{k+h}]$.
\end{proof}
\begin{remark}
It is not difficult to see that this algorithm is applicable to leader-follower networks, in which the leader receives no information from others. For example, if agent $i$ is the leader, then $l_{ii}=0$ and the update time interval $\Delta t_{i}^{0}=+\infty$ by definition, which mean the leader $i$ does not need to update its state. Generally, if agent $i$ wants a long update time interval, all it needs to do is to choose small coupling weights so that $l_{ii}$ is small, and this can be done independent of all other agents.
\end{remark}

From the above analysis, it is not difficult to see that the distributed event-triggered algorithm can also be applied to networks with switching topologies. For example, at each update time $t_{i}^{k}$, agent $i$ may reset its coupling weights $l_{ij}$ such that the algorithm can be written in the form:
\begin{eqnarray}\label{eqnDistributedSwitching}
  \dot{x}_{i}(t)=-\sum_{j=1}^{n}l_{ij}^{k}x_{j}(t_{j}^{k_{j}(t)}),~~t\in [t_{i}^{k},t_{i}^{k+1}),~i=1,2,\cdots,n,
\end{eqnarray}
where $\delta_{i}/l_{ii}^{k}\le\Delta t_{i}^{k}\le (1-\delta_{i})/l_{ii}^{k}$. One of the great advantage of this algorithm is that each agent can adjust its update time interval independently based on its own need. If agent $i$ wants a longer update time interval, then it may decrease the coupling weight, otherwise, increase the coupling weight. For example, each agent $i$ may store a scaling factor $\epsilon_{i}^{k}$ such that
\begin{eqnarray}\label{eqnWeightScaling}
l_{ij}^{k}=\epsilon_{i}^{k}l_{ij}.
\end{eqnarray}
By a similar analysis as in the case of fixed topology, we can prove
\begin{theorem}
Suppose that $\epsilon_{\min}$ and $\epsilon_{\max}$ are two positive number satisfying $0<\epsilon_{\min}<\epsilon_{\max}$ and  for all $i$, $k$, $\epsilon_{\min}\le\epsilon_{i}^{k}\le \epsilon_{\max}$. If the graph $\mathcal{G}(L)$ has a spanning tree,
then the distributed event-triggered consensus algorithm with switching topology \eqref{eqnDistributedSwitching} and weight updating rule \eqref{eqnWeightScaling} can reach consensus as $t\to \infty$.
\end{theorem}

Finally, we investigate event-triggered consensus algorithm in networks with stochastically switching topologies in which each agent $i$ updates its coupling weights independently at the same time it updates its state.

Given $0<a<b$, and let $S_{i}(a,b)=\{s=[s_{1},\cdots,s_{n}]:~s_{i}\in [a,b], s_{j}\le 0, j\ne i, \sum_{j=1}^{n}s_{j}=0\}$. The event-triggered algorithm can be formulated as follows:
\begin{eqnarray}\label{eqnDistributedIIdSwitching}
  \dot{x}_{i}(t)=-\sum_{j=1}^{n}\tilde{l}_{ij}^{k}x_{j}(t_{j}^{k_{j}(t)}),~t\in [t_{i}^{k},t_{i}^{k+1}),
\end{eqnarray}
where $\tilde{L}_{i}^{k}=[\tilde{l}_{i1}^{k},\cdots,\tilde{l}_{in}^{k}]\in S_{i}(a_{i},b_{i})$ for some given $0<a_{i}<b_{i}$, and $\{\tilde{L}_{i}^{k}\}$ is an independent and identically distributed sequence.
Before stating the convergence result, we summarize the following assumption.
\begin{assumption}\label{assumDistributedIIdSwitching}
\begin{enumerate}
  \item[(i)] There exist $\delta_{i}\in (0,1/2)$ such that $\Delta t_{i}^{k}\in [\delta_{i}/\tilde{l}_{ii}^{k},(1-\delta_{i})/\tilde{l}_{ii}^{k}]$;

  \item[(ii)] There exists $\delta>0$ such that the graph with Laplacian being $\E \tilde{L}^{k}$ has a $\delta$-spanning tree, where $\tilde{L}^{k}=[\tilde{L}_{1}^{k\top},\cdots,\tilde{L}_{n}^{k\top}]^{\top}$ is a matrix with its $i$th row being $\tilde{L}_{i}^{k}$;
\end{enumerate}
\end{assumption}

\begin{theorem}\label{thmDistributedIIdSwitching}
Under Assumption \ref{assumDistributedIIdSwitching}, the event-triggered consensus algorithm \eqref{eqnDistributedIIdSwitching} will reach a consensus almost surely.
\end{theorem}

\begin{proof}
The algorithm \eqref{eqnDistributedIIdSwitching} can be reformulated as:

\begin{eqnarray}\label{eqnDistributedDependentSwitching}
  \dot{x}_{i}(t)=-\sum_{j=1}^{n}l_{ij}^{k}x_{j}(t_{j}^{k_{j}(t)}),~t\in [t_{k},t_{k+1}),
\end{eqnarray}
where $l_{ij}^{k}=\tilde{l}_{ij}^{k_{i}(t_{k})}$.
Let $L^{k}=[l_{ij}^{k}]$, then generally the sequence $\{L^{k}\}$ is not an independent sequence since for each $k$, $l_{ij}^{k}=l_{ij}^{k+1}$ if no update of agent $i$ occurs at time $t_{k+1}$. However, similar to the analysis given above, it can be seen that $L^{k}$ is independent of $L^{k'}$ for $k'\ge k+N$, where $N=n(h'+1)$ and
$h'$ is the smallest positive integer satisfying $h'\min_{i}\{\delta_{i}/b_{i}\}\ge \max_{i}\{(1-\delta_{i})/a_{i}\}$.

Similarly, we have
\begin{eqnarray}
  y(k+1)=\sum_{l=0}^{N}\sum_{j=1}^{n}a_{ij}^{l}(k)y_{j}(k-l),
\end{eqnarray}
where
$a_{ij}^{l}(k)\ge 0$ for all $i,j,l,k$, and
$$
\sum_{l=0}^{N}\sum_{j=1}^{n}a_{ij}^{l}(k)=1,
$$
for all $i$, $k$.
Since $a_{ij}^{l}(k)=-\Delta t_{k-l}l_{ij}^{k-l}$, we can see that $A^{0}(k),\cdots,A^{N}(k)$ are independent of $A^{0}(k'),\cdots,A^{N}(k')$, when
$k'\ge k+2N$. On the other hand, in each time interval $(t_{k},t_{k+N})$, each agent there at least updates once. Thus we have
$$
\sum_{m=k+1}^{k+N}\E B(k)\ge \delta'\tilde{L}_{0},
$$
where $B(k)$ is defined as before and $\delta'=\min_{i,k}\{\Delta t_{i}^{k}\}=\min_{i}\{\delta_{i}/b_{i}\}>0$, and $[\tilde{L}_{0}]_{ij}=-[\E\tilde{L}^{k}]_{ij}$ for $i\ne j$ and $[\tilde{L}_{0}]_{ii}=0$ for all $i$. This implies that $\sum_{m=k+1}^{k+N}\E B(k)$ has a $\delta''$-spanning tree with $\delta''=\delta\delta'>0$. The conclusion follows from Corollary \ref{corollaryAppendixIndependent} in the Appendix.

\end{proof}

\section{Numerical Simulation}\label{secSimulation}

In this section, we provide an example to illustrate the theoretical results in the previous section. We use the distributed event-triggered algorithm in networks with stochastically switching topologies as considered in Theorem \ref{thmDistributedIIdSwitching}.

Consider a network of four agents with
\begin{align*}
&\tilde{L}_{1}^{k}\in \{[1, -1 ,0 ,0], [1,0, 0,-1]\}, \\
&\tilde{L}_{2}^{k}=\{[-1,1,0,0], [0,1,-1,0]\},\\
&\tilde{L}_{3}^{k}\in \{[0,-1,1,0], [0,0,1,-1]\},\\
&\tilde{L}_{4}^{k}\in \{[0,-1,0,1], [0,0,-1,1]\},
\end{align*}
and each agent selects its coupling weights using a uniform distribution. We choose $\delta_{i}=0.1$ for all the agents. The next update time is randomly chosen from the permissible range. It is not difficult to verify that the conditions in Theorem \ref{thmDistributedIIdSwitching} are satisfied. The simulation results are provided in Fig. \ref{figIIdSwitching} with the initial value of the four agents being randomly chosen. It can be seen that the agents actually reached consensus.

\begin{figure}
  \centering
  \includegraphics[width=0.8\textwidth]{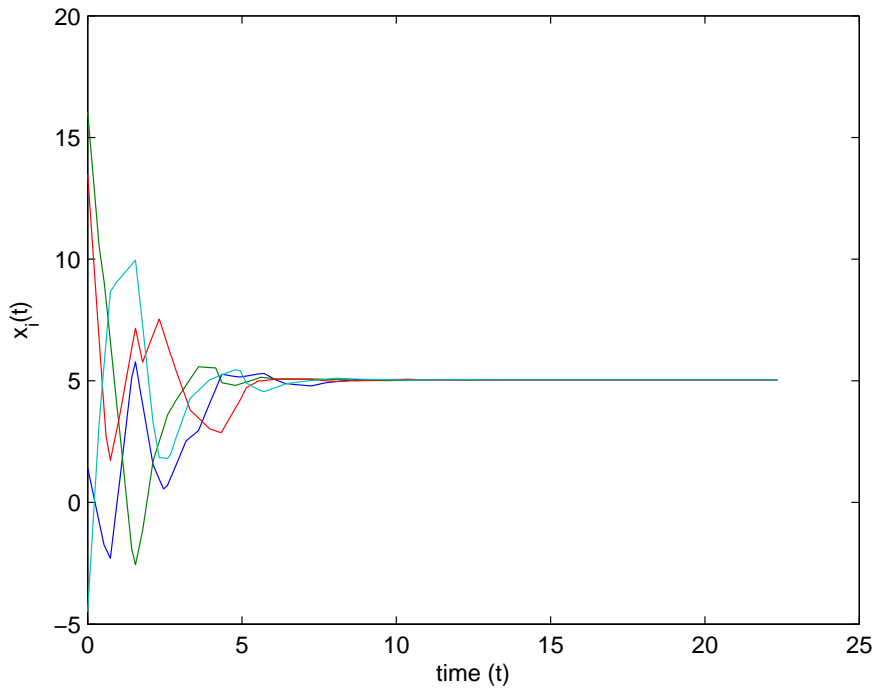}
  \caption{Self-triggered consensus in networks of four agents with stochastically switching topologies.}\label{figIIdSwitching}
\end{figure}

\section{Conclusion}\label{secConclusion}

In this paper, we proposed a new structure-based self-triggered consensus algorithm in both centralized approach and distributed approach. Different from existing works that used a quadratic Lyapunov function as the analysis tool, we reduce the self-triggered consensus algorithm to a discrete-time consensus algorithm by some proper transformation, which enables the application of the product theory of stochastic matrices to the convergence analysis. Compared to existing work, our method provides several advantages. First, each agent does not need to calculate the system error to determine its update time. Second, it provides explicit positive lower and upper bounds for the update interval of each agent based on its coupling weights, which can be adjusted by each agent independent of other agents. We also used our method to investigate networks with switching topologies, especially networks with stochastically switching topologies. Our work reveals that the event/self-triggered algorithms are essentially discrete and thus more suitable to a discrete analysis framework. And it is predictable that more results can be obtained for self-triggered algorithms under our new analysis framework in the future.

\section*{Appendix: Discrete-time consensus analysis}

The appendix is devoted to the discussion of discrete-time consensus algorithms in networks of multi-agents with time delays based on the product theory of stochastic matrices and provide some results that is used in the main text.

\subsection{Preliminaries}

This subsection will provide some preliminaries on stochastic matrices that will be used in the proof of Theorem \ref{thmAppendixMain} in the next subsection.

A nonnegative matrix $A$ is called a {\em stochastic indecomposable and aperiodic (SIA)} if it is a stochastic matrix and there exists a column vector $v$ such that $\lim_{k\to\infty}A^{k}={\bf e}_{n}v^{\top}$, where ${\bf e}_{n}$ is an $n$-dimensional column vector with all entries $1$. A sequence of stochastic matrices $\{A_{k}\}$ is {\em ergodic} if there exists a column vector $v$ such that $\lim_{k\to\infty}\prod_{k=1}^{\infty}A_{k}={\bf e}_{n}v^{\top}$. $A$ is called {\em $\delta$-SIA} if its $\delta$-matrix is SIA. A stochastic matrix $A$ is called {\em scrambling} if for any pair ($i$, $j$), there exists an index $k$ such that both $a_{ik}>0$ and $a_{jk}>0$. Similarly, $A$ is called {\em $\delta$-scrambling} if its $\delta$-matrix is scrambling. For a matrix $A$, $\diag(A)$ refers to the diagonal matrix formed by the diagonal elements of $A$, i.e., $[\diag(A)]_{ii}=[A]_{ii}$ for all $i$ and $[\diag(A)]_{ij}=0$ for $i\ne j$. And we define $\overline{\diag}(A)=A-\diag(A)$.

We use $\Prb\{\cdot|\mathcal{F}\}$ to denote the conditional probability with respect to $\mathcal{F}$ on a probability space $\{\Omega,\Prb,\mathcal{F}\}$, which can also be expressed into the conditional expectation of an indicator function, i.e.,
$\Prb\{\cdot|\mathcal{F}\}=\E\{{\bf 1}_{\{\cdot\}}|\mathcal{F}\}$, where ${\bf 1}_{\{\cdot\}}$ is an indicator function which is defined as
\begin{eqnarray*}
  {\bf 1}_{S}(x)=\left\{
  \begin{array}{cc}
  1, & x\in S,\\
  0, & x\not\in S.
  \end{array}
  \right.
\end{eqnarray*}
Thus, $\Prb\{\cdot|\mathcal{F}\}$ is also a random variable that is measurable with respect to $\mathcal{F}$.

In the following, we will introduce some lemmas that will be used in the proof of the main results.

Let $A=[a_{ij}]$ be an $n\times n$ stochastic matrix. Define
\begin{eqnarray*}
\Delta(A)=\max_{j}\max_{i_{1},i_{2}}|a_{i_{1}j}-a_{i_{2}j}|.
\end{eqnarray*}
It can be seen that $\Delta(A)$ measures how different the rows of $A$ are. $\Delta(A)=0$ if and only if the rows of $A$ are identical.
Define
\begin{eqnarray*}
\lambda(A)=1-\min_{i_{1},i_{2}}\sum_{j}\min\{a_{i_{1}j},a_{i_{2}j}\}.
\end{eqnarray*}
\begin{remark}
It can be seen that if $A$ is scrambling, then $\lambda<1$, if $A$ is $\delta$-scrambling for some $\delta>0$, then $\lambda(A)\le 1-\delta$.
\end{remark}

\begin{lemma}\cite{Wolfowitz}\label{lemContractionOfProduct}
For any stochastic matrices $A_{1}$, $A_{2}$, $\cdots$, $A_{k}$, $k>0$,
\begin{eqnarray*}
\Delta(A_{1}A_{2}\cdots A_{k})\le \prod_{i=1}^{k}\lambda(A_{i}).
\end{eqnarray*}
\end{lemma}
\begin{remark}\label{remErgodicityOfProducts}
From Lemma \ref{lemContractionOfProduct}, it can be seen that for a sequence of stochastic matrices $\{A_{k}\}$, if there exists a sequence $\{k_{i}\}$ such that $k_{i}<k_{i+1}$, $\lim_{i\to\infty}k_{i}=\infty$, and $\prod_{i=1}^{\infty}\lambda(\prod_{m=k_{i}+1}^{k_{i+1}}A_{m})=0$, then $\lim_{k\to\infty}\Delta(\prod_{m=1}^{k} A_{m})=0$, i.e., $\{A_{k}\}$ is ergodic. Particularly, from the property of scrablingness, if there exists $\delta>0$ such that there are infinite matrices in the sequence $\{\prod_{m=k_{i}+1}^{k_{i+1}}A_{m}\}$ that are $\delta$-scrambling, then $\{A_{k}\}$ is ergodic.
\end{remark}

\begin{lemma}\cite{Wolfowitz,XiaoWang_TAC_2008}\label{lemSIAtoScrambling}
Let $A_{1}$, $A_{2}$, $\cdots$, $A_{k}$ (repetitions permitted) be $r\times r$ SIA matrices with the property that for any $1<k_{1}<k_{2}\le k$, $\prod_{i=k_{1}}^{k_{2}}A_{i}$ is SIA. If $k>{\bf n}(r)$, then $\prod_{i=1}^{k}A_{i}$ is a scrambling matrix. Here ${\bf n}(r)$ is the number of different types of all $r\times r$ stochastic matrices.
\end{lemma}

\begin{lemma}\cite{XiaoWang_TAC_2008}\label{lemHighDimensionSpanningTree}
Let $A_{1}$, $\cdots$, $A_{m}$ be $n\times n$ nonnegative matrices, let
\begin{eqnarray*}
D=\left[
\begin{array}{cccc}
A_{1}& A_{2}& \cdots & A_{m}\\
0& 0& \cdots & 0\\
\cdots & \cdots & \ddots & \cdots\\
0 & 0& \cdots &0
\end{array}
\right]_{mn\times mn},
\end{eqnarray*}
let
\begin{eqnarray*}
M_{0}=\left[
\begin{array}{ccccc}
I & 0& \cdots& 0&0\\
I& 0&\cdots& 0&0\\
0&I&\cdots& 0&0\\
\vdots& \vdots&\ddots& 0&0\\
0&0&\cdots& I&0
\end{array}
\right]_{mn\times mn},
\end{eqnarray*}
and let $M_{k}=D+M_{0}^{k}$ for any $k\in \{1,2,\cdots,m-1\}$. Then if $\mathcal{G}(\sum_{i=1}^{m}A_{i})$ contains a spanning tree, then $\mathcal{G}(M_{k})$ contains a spanning tree with the property that the root vertex of the spanning tree has a self-loop in $\mathcal{G}(M_{k})$.
\end{lemma}
\begin{remark}
Actually, since $M_{0}^{k}=M_{0}^{m-1}=\left[\begin{array}{cccc}I & 0 & \cdots & 0 \\
I & 0 & \cdots & 0 \\ \vdots & \vdots & \ddots & \vdots \\
I & 0 & \cdots & 0\end{array}\right]$ for all $k\ge m-1$, thus Lemma \ref{lemHighDimensionSpanningTree} holds for all $k\in \mathbb{N}$. This is important for our further results based on this lemma.
\end{remark}


\begin{remark}
If the condition is $\mathcal{G}(\sum_{i=1}^{m}A_{i})$ has a $\delta$-spanning tree for some $\delta>0$, then the conclusion becomes that $\mathcal{G}(M_{k})$ has a $\delta'$-spanning tree whose root vertex has a self-loop, where $\delta'\ge\delta/m$.
\end{remark}

\begin{lemma}\cite{WangL_JCtrl_2006}\label{lemSpanningTreeToSIA}
Let $A$ be a stochastic matrix. If $\mathcal{G}(A)$ contains a spanning tree with the property that the root vertex of the spanning tree has a self-loop in $\mathcal{G}(A)$, then A is SIA.
\end{lemma}

From Lemma \ref{lemHighDimensionSpanningTree} and \ref{lemSpanningTreeToSIA}, we can have the following Lemma.
\begin{lemma}\label{lemProdHighDimensionSpanningTree}
Let $A_{1}^{i}$, $\cdots$, $A_{m}^{i}$, $i=1,\cdots,k$ be $n\times n$ nonnegative matrices. Let
\begin{eqnarray*}
D_{i}=\left[
\begin{array}{cccc}
A_{1}^{i}& A_{2}^{i}& \cdots & A_{m}^{i}\\
0& 0& \cdots & 0\\
\cdots & \cdots & \ddots & \cdots\\
0 & 0& \cdots &0
\end{array}
\right]_{mn\times mn},
\end{eqnarray*}
and $M_{0}$ being defined as in Lemma \ref{lemHighDimensionSpanningTree}. Then if $\mathcal{G}(\sum_{i=1}^{k}\sum_{j=1}^{m}A_{j}^{i})$ has a $\delta$-spanning tree for some $\delta>0$, then the product
$\prod_{i=1}^{k}(D_{i}+M_{0})$ has a $\delta'$-spanning tree for some $0<\delta'<\delta$. And the root of the spanning tree has a self-loop, thus $\prod_{i=1}^{k}(D_{i}+M_{0})$ is $\delta'$-SIA.
\end{lemma}

\begin{remark}\label{remProdHighDimensionSpanningTree}
It is obvious from Lemma \ref{lemHighDimensionSpanningTree} that if there exists $\epsilon>0$ such that $D'_{i}\ge \epsilon (M_{0}+D_{i})$, then $\prod_{i=1}^{k}D_{i}'$ is $\delta'$-SIA for some $0<\delta'<\delta$. This is the case that will appear in the proof of Theorem \ref{thmAppendixMain}.
\end{remark}

\begin{proof}First, we have
\begin{eqnarray}\label{eqnProductSIA}
\prod_{i=1}^{k}(M_{0}+D_{i})\ge M_{0}^{k}+\sum_{i=1}^{k}M_{0}^{k-i}D_{i}M_{0}^{i-1}\ge M_{0}^{k}+\sum_{i=1}^{k}D_{i}M_{0}^{i-1},
\end{eqnarray}
where the second inequality is due to the fact that $M_{0}^{j}D_{i}\ge D_{i}$ for any $j\ge 0$. And it is not difficult to verify that the first block row sum of $D_{i}$ is preserved in $D_{i}M_{0}^{i-1}$ for $i=1,\cdots,k$. Thus the first block row sum of $\sum_{i=1}^{k}D_{i}M_{0}^{i-1}$ equals that of $\sum_{i=1}^{k}D_{i}$, i.e., $\sum_{i=1}^{k}\sum_{j=1}^{m}A_{j}^{i}$. This means $\mathcal{G}(\sum_{i=1}^{k}D_{i}M_{0}^{i-1})$ has a $\delta$-spanning tree. From Lemma \ref{lemHighDimensionSpanningTree}, $\mathcal{G}(M_{0}^{k}+\sum_{i=1}^{k}D_{i}M_{0}^{i-1})$ has a $\delta'$-spanning tree for some $0<\delta'<\delta$, and the root has a self-loop. This is also true for  the product $\prod_{i=1}^{k}(M_{0}+D_{i})$ due to \eqref{eqnProductSIA}. Thus, $\prod_{i=1}^{k}(M_{0}+D_{i})$ is $\delta'$-SIA from Lemma \ref{lemSpanningTreeToSIA}. The proof is completed.
\end{proof}
From Lemma \ref{lemProdHighDimensionSpanningTree}, we can easily obtain the following corollary.
\begin{corollary}\label{corollaryProdHighDimensionSpanningTree}
Let $A_{1}^{i}$, $\cdots$, $A_{m}^{i}$, $i=1,2,\cdots,$ be uniformly bounded $n\times n$ nonnegative matrices. Let $D_{i}$ and $M_{0}$ be defined as in Lemma \ref{lemProdHighDimensionSpanningTree}. If there exists $k$ such that $\mathcal{G}(\sum_{i=1}^{k}\sum_{j=1}^{m}A_{j}^{i})$ has a $\delta$-spanning tree for some $\delta>0$, then for each $k'>k$, the product
$\prod_{i=1}^{k'}(D_{i}+M_{0})$ is $\delta'(k')$-SIA, where $0<\delta'(k')<\delta$ depends on $k'$.
\end{corollary}
\begin{proof}
Since $A_{j}^{i}\ge 0$, if $\mathcal{G}(\sum_{i=1}^{k}\sum_{j=1}^{m}A_{j}^{i})$ has a $\delta$-spanning tree for some $\delta>0$, then for each $k'>k$, $\mathcal{G}(\sum_{i=1}^{k'}\sum_{j=1}^{m}A_{j}^{i})$ also has a $\delta$-spanning tree. The conclusion is obvious from Lemma \ref{lemProdHighDimensionSpanningTree}.
\end{proof}
\begin{remark}\label{remProductSIA}
If $A_{1}^{i}$, $\cdots$, $A_{m}^{i}$, $i=1,2,\cdots,$ is a random sequence, then it is obvious from Corollary \ref{corollaryProdHighDimensionSpanningTree} that the event $\mathcal{G}(\sum_{i=1}^{k}\sum_{j=1}^{m}A_{j}^{i})$ has a $\delta$-spanning tree is contained in the event $\prod_{i=1}^{k'}(D_{i}+M_{0})$ is $\delta'(k')$-SIA.
\end{remark}

\begin{lemma}[second Borel--Cantelli lemma \cite{Durrett}]\label{lemSecondBorelCantelli}
Let $\{\mathcal{F}_{n}, n\ge 0\}$ be a filtration with
$\mathcal{F}_{0}=\{\emptyset,\Omega\}$ and $\{X_{n}, n\ge 1\}$ a
sequence of events with $X_{n}\in\mathcal{F}_{n}$. Then
\begin{eqnarray*}
\{X_{n}~\textnormal{ \it occurs~infinitely~often
}\}=\Big\{\sum_{n=1}^{+\infty}\Prb\{X_{n}|\mathcal{F}_{n-1}\}=+\infty\Big\},
\end{eqnarray*}
with a probability $1$, where ``infinitely often" means that an
infinite number of events from $\{X_{n}\}_{n=1}^{\infty}$ occur.
\end{lemma}

\begin{lemma}\label{lemProbabilityToLongerSeqence}
Let $\{A_{k},\mathcal{F}_{k}\}$ be an adapted random sequence. If for a sequence $\{X_{k}\}$ defined as $X_{k}=\{A_{k}\in S_{k}\}$ for some given set $S_{k}$ there exists $\delta>0$ such that the conditional probability
$\Prb\{X_{m+1}|\mathcal{F}_{m}\}>\delta$, then for each $m, h>0$,
\begin{eqnarray*}
\Prb\{X_{m+1},\cdots,X_{m+h}|\mathcal{F}_{m}\}>\delta^{h}.
\end{eqnarray*}
\end{lemma}
\begin{proof}
By the definition of conditional probability, we have
\begin{eqnarray*}
&&\Prb\{X_{m+1},\cdots,X_{m+h}|\mathcal{F}_{m}\}\\
&=&\E\{{\bf 1}_{X_{m+1}}\cdots{\bf 1}_{X_{m+h}}|\mathcal{F}_{m}\}\\
&=&\E\{{\bf 1}_{X_{m+1}}\cdots\E\{{\bf 1}_{X_{m+h}}|\mathcal{F}_{m+h-1}\}|\mathcal{F}_{m}\}\\
&\ge&\delta\E\{{\bf 1}_{X_{m+1}}\cdots{\bf 1}_{X_{m+h-1}}|\mathcal{F}_{m}\}\\
&\ge&\delta^{2}\E\{{\bf 1}_{X_{m+1}}\cdots{\bf 1}_{X_{m+h-2}}|\mathcal{F}_{m}\}\\
&\ge&\cdots\\
&\ge&\delta^{h}.
\end{eqnarray*}
\end{proof}

\begin{lemma}(Lemma 5.4 in \cite{LiuB_SICON_2011})\label{lemExpectationToProbability1}
Let $\{\Omega, \mathcal{F},\Prb\}$ be a probability space, and $f$ be a
random variable with $0\le f\le 1$. If for a $\sigma$-algebra
$\mathcal{F}'\subseteq \mathcal{F}$, a set $S\in \mathcal{F}'$ with
$\Prb\{S\}>0$, $\E\{f|\mathcal{F}'\}\ge\delta$ holds on $S$ for some
$\delta>0$, then we have
\begin{eqnarray*}
\Prb\{f\ge\frac{\delta}{2}|\mathcal{F}'\}\ge\frac{\delta}{2},
\end{eqnarray*}
on $S$ and particularly,
\begin{eqnarray*}
\E\{f{\bf
1}_{\{f\ge\frac{\delta}{2}\}}|\mathcal{F}'\}\ge\frac{\delta^{2}}{4}
\end{eqnarray*}
on $S$.
\end{lemma}

\subsection{Discrete-time delayed consensus analysis}

In this section, we will discuss discrete-time consensus algorithms in networks with switching topologies and time delays. We first consider general stochastic processes described by an adapted sequence and establish a most general result. Then we use it to obtain two corollaries under two special cases that will appear in the main text.

Consider the following discrete time dynamical systems with switching topologies:
\begin{eqnarray}\label{eqnDiscrete}
  x_{i}(k+1)=\sum_{l=0}^{\tau}\sum_{j=1}^{n}a_{ij}^{l}(k)x_{j}(k-l),~~i=1,\cdots,n,
\end{eqnarray}
where $k$ is the time index, $x(k)=[x_{1}(k),\cdots,x_{n}(k)]^{\top}\in \mathbb{R}^{n}$ is the state variable of the system at time $k$, $\tau\ge 0$ is the bound of the time delay, where $\tau=0$ corresponds to the case of no time delay.
In the following, we always assume that $a_{ii}^{0}(k)>0$, $a_{ii}^{l}(k)=0$ for $l\ne 0$, $a_{ij}^{l}(k)\ge 0$ for each $i$, $j$, $l$, $k$, and
\begin{eqnarray*}
    \sum_{l=0}^{\tau}\sum_{j=1}^{n}a_{ij}^{l}(k)=1
\end{eqnarray*}
holds for each $k$ and $i$.

Define $A^{l}(k)=[a_{ij}^{l}(k)]$, and $B(k)=[b_{ij}(k)]=\sum_{l=0}^{\tau}A^{l}(k)$.

We have the following Theorem.

\begin{theorem}\label{thmAppendixMain}
  Let $\{A^{0}(k),\cdots,A^{\tau}(k),\mathcal{F}_{k}\}$ be an adapted process. If there exists $\delta>0$, $h>0$ such that $B(k)\ge \delta I$ and the conditional expectation $\E\{\sum_{k=m+1}^{m+h}B(k)|\mathcal{F}_{m}\}$ has a $\delta$-spanning tree, then the system (\ref{eqnDiscrete}) will reach consensus almost surely.
\end{theorem}

\begin{remark}
When $\tau=0$, the system has no time delays,and Theorem \ref{thmAppendixMain} reduces to Theorem 3.1 of \cite{LiuB_SICON_2011}. From this point of view, this result can be seen as an extension of that in \cite{LiuB_SICON_2011}.
\end{remark}

The proof of this theorem will be divided into several steps.

First, we will prove the following lemma, which shows that the conditional expectation of a spanning tree implies a positive conditional probability of existence of a spanning tree for a longer length of the summation.

\begin{lemma}\label{lemExpectationToProbability}
Let $\{A_{k},\mathcal{F}_{k}\}$ be an adapted random sequence of $n\times n$ stochastic matrices, if there exists $\delta>0$ such that for each $m$, $\E\{A_{m+1}|\mathcal{F}_{m}\}$ has a $\delta$ -spanning tree, then there exist $h>0$, $0<\delta'<\delta$ such that
\begin{eqnarray}\label{eqnProbabilitySpanningTree}
\Prb\{\sum_{k=m+1}^{m+h}A_{k} \textnormal{ has a $\delta'$-spanning tree }|\mathcal{F}_{m}\}>\delta'.
\end{eqnarray}
\end{lemma}
\begin{proof}
Denote $\Sp(n)$ the number of different types of spanning trees composed of $n$ vertices. For some fixed $m$, we decompose the probability space $\Omega$ with respect to $\mathcal{F}_{m}$, i.e., $S_{m}^{(i_{1})}\in \mathcal{F}_{m}$, $\Omega=\cup S_{m}^{(i_{1})}$, $i_{1}=1,2,\cdots,s_{1}$ where $s_{1}\le \Sp(n)$, and $S_{m}^{(i)}\cap S_{m}^{(j)}=\emptyset$ for $i\ne j$ such that on each $S_{m}^{(i)}$, $\E\{A_{m+1}~|~\mathcal{F}_{m}\}$ has a (fixed) specified spanning tree. Denote $\E\{A_{m+1}~|~\mathcal{F}_{m}\}_{S_{m}^{(i_{1})}}$ the restriction of $\E\{A_{m}~|~\mathcal{F}_{m}\}$ on ${S_{m}^{(i_{1})}}$, i.e.,
\begin{eqnarray*} \E\{A_{m+1}~|~\mathcal{F}_{m}\}_{S_{m}^{(i_{1})}}=\E\{A_{m+1}~|~\mathcal{F}_{m}\}\times{\bf 1}_{S_{m}^{(i_{1})}}.
\end{eqnarray*}

Let $\ST_{\delta}(\E\{A_{m+1}~|~\mathcal{F}_{m}\}_{S_{m}^{(i_{1})}})$ be a $\delta$-spanning tree contained in $\E\{A_{m+1}~|~\mathcal{F}_{m}\}_{S_{m}^{(i_{1})}}$, which is arbitrarily selected when more than one choices are available. For each $S_{m}^{(i_{1})}$, we pick an edge $(j,i)\in \mathcal{E}(\ST_{\delta}(\E\{A_{m+1}~|~\mathcal{F}_{m}\}_{S_{m}^{(i_{1})}}))$,
and let $S_{m+1}^{(i_{1})}=\{[A_{m+1}]_{ij}\ge \delta/2\}\cap S_{m}^{(i_{1})}$. Then $S_{m+1}^{(i_{1})}\in \mathcal{F}_{m+1}$, and from Lemma \ref{lemExpectationToProbability1},
\begin{eqnarray*}
\Prb\{S_{m+1}^{(i_{1})}~|~\mathcal{F}_{m}\}\ge \frac{\delta}{2}
\end{eqnarray*}
holds on $S_{m}^{(i_{1})}$.

Similarly, we decompose each $S_{m+1}^{(i_{1})}$ with respect to $\mathcal{F}_{m+1}$, i.e., $S_{m+1}^{(i_{1})}=\cup_{i_{2}}S_{m+1}^{(i_{1}i_{2})}$ with $S_{m+1}^{(i_{1}i_{2})}\in\mathcal{F}_{m+1}$ and $S_{m+1}^{(i_{1}i_{2})}\cap S_{m+1}^{(i_{1}i_{2}')}$ for $i_{2}, i_{2}'=1,2,\cdots,s_{2}$, $i_{2}\ne i_{2}'$, where $s_{2}\le \Sp(n)$ depends on the index $i_{1}$ such that $\E\{A_{m+2}~|~\mathcal{F}_{m+1}\}_{S_{m+1}^{(i_{1}i_{2})}}$ has a specified $\delta$-spanning tree on each $S_{m+1}^{(i_{1}i_{2})}$.

For each $S_{m+1}^{(i_{1}i_{2})}$, we pick an edge $(j,i)\in \mathcal{E}(\ST_{\delta}(\E\{A_{m+2}~|~\mathcal{F}_{m+1}\}_{S_{m+1}^{(i_{1}i_{2})}}))$ and let $S_{m+2}^{(i_{1}i_{2})}=\{[A_{m+2}]_{ij}\ge \delta/2\}\cap S_{m+1}^{(i_{1}i_{2})}$, then $S_{m+2}^{(i_{1}i_{2})}\in \mathcal{F}_{m+2}$ and from Lemma \ref{lemExpectationToProbability1},
\begin{eqnarray*}
\Prb\{S_{m+2}^{(i_{1}i_{2})}~|~\mathcal{F}_{m+1}\}\ge \frac{\delta}{2}
\end{eqnarray*}
holds on $S_{m+1}^{(i_{1}i_{2})}$.

Continuing this process, we can get a sequence:
\begin{eqnarray*}
S_{m}^{(i_{1})},S_{m+1}^{(i_{1})}, S_{m+1}^{(i_{1}i_{2})},S_{m+2}^{(i_{1}i_{2})},\cdots,S_{m+k-1}^{(i_{1}i_{2}\cdots i_{k})}, S_{m+k}^{(i_{1}i_{2}\cdots i_{k})}
\end{eqnarray*}
such that for $l=1,2,\cdots,k$, $S_{m+l}^{(i_{1}i_{2}\cdots i_{l})}\in \mathcal{F}_{m+l}$, $S_{m+l-1}^{(i_{1}i_{2}\cdots i_{l-1})}=\cup_{i_{l}} S_{m+l-1}^{(i_{1}i_{2}\cdots i_{l})}$, $\E\{A_{m+l}~|~\mathcal{F}_{m+l-1}\}_{S_{m+l-1}^{(i_{1}i_{2}\cdots i_{l})}}$ has a (fixed) spanning tree, and
\begin{eqnarray}\label{eqnChooseEdge}
S_{m+l}^{(i_{1}i_{2}\cdots i_{l})}=\{[A_{m+l}]_{ij}>\delta/2\}\cap S_{m+l-1}^{i_{1}i_{2}\cdots i_{l}},
\end{eqnarray}
where the edge $(j,i)\in \mathcal{E}(\ST_{\delta}(\E\{A_{m+l}~|~\mathcal{F}_{m+l-1}\}_{S_{m+l-1}^{(i_{1}i_{2}\cdots i_{l})}}))$.

For any fixed sequence $i_{1},i_{2},\cdots, i_{k}$, $S_{m+k}^{(i_{1}i_{2}\cdots i_{k})}\subseteq S_{m+k-1}^{(i_{1}i_{2}\cdots i_{k})}\subseteq S_{m+k-1}^{(i_{1}i_{2}\cdots i_{k-1})}\subseteq \cdots \subseteq S_{m+1}^{(i_{1})}$. We choose the edge $(j,i)$ in (\ref{eqnChooseEdge}) for each $S_{m+l}^{(i_{1}i_{2}\cdots i_{l})}$ in such way that if $\ST_{\delta}(\E\{A_{m+1}~|~\mathcal{F}_{m+l-1}\}_{S_{m+l-1}^{(i_{1}i_{2}\cdots i_{l})}})$ hasn't appeared in the sequence, $\ST_{\delta}(\E\{A_{m+1}~|~\mathcal{F}_{m}\}_{S_{m}^{(i_{1} )}})$, $\cdots$, $\ST_{\delta}(\E\{A_{m+1-1}~|~\mathcal{F}_{m+l-2}\}_{S_{m+l-2}^{(i_{1}i_{2}\cdots i_{l-1})}})$, then we choose $(j,i)\in \mathcal{E}(\ST_{\delta}(\E\{A_{m+1}~|~\mathcal{F}_{m+l-1}\}_{S_{m+l-1}^{(i_{1}i_{2}\cdots i_{l})}}))$ arbitrarily. Otherwise, we choose $(j,i)\in \mathcal{E}(\ST_{\delta}(\E\{A_{m+1}~|~\mathcal{F}_{m+l-1}\}_{S_{m+l-1}^{(i_{1}i_{2}\cdots i_{l})}}))$ which hasn't been chosen before when \\ $\ST_{\delta}(\E\{A_{m+1}~|~\mathcal{F}_{m+l-1}\}_{S_{m+l-1}^{(i_{1}i_{2}\cdots i_{l})}})$ appears in the sequence \\ $\ST_{\delta}(\E\{A_{m+1}~|~\mathcal{F}_{m}\}_{S_{m}^{(i_{1})}})$, $\cdots$, $\ST_{\delta}(\E\{A_{m+1-1}~|~\mathcal{F}_{m+l-2}\}_{S_{m+l-2}^{(i_{1}i_{2}\cdots i_{l-1})}})$ if there still exists such an edge.

Since there at most $\Sp(n)$ different types of spanning trees, and each spanning tree has $n-1$ edges, it is obvious when $k=(n-1)\Sp(n)$, the sets
\begin{eqnarray*}
S_{m+k}^{(i_{1}i_{2}\cdots i_{k})}\subseteq \{\sum_{l=1}^{k}A_{m+l} \textnormal{ has $\delta/2$-spanning tree} \}.
\end{eqnarray*}
This is because for each fixed sequence $i_{1}$, $\cdots$, $i_{k}$, at least one type of spanning trees has appeared at least $n-1$ times, thus by the edge chosen strategy, each of its edges has been chosen at least once.

So we have
\begin{eqnarray}
&&\Prb\{\sum_{l=1}^{k}A_{m+l} \textnormal{ has $\delta/2$-spanning tree}~|~\mathcal{F}_{m}\}\\
&\ge & \Prb\{\cup_{i_{1},i_{2},\cdots,i_{k}}S_{m+k}^{(i_{1}i_{2}\cdots i_{k})}~|~\mathcal{F}_{m}\}\nonumber\\
&=&\E\{\sum_{i_{1},i_{2},\cdots,i_{k}}{\bf 1}_{S_{m+k}^{(i_{1}i_{2}\cdots i_{k})}}~|~\mathcal{F}_{m}\}\nonumber\\
&=&\E\{\E\{\sum_{i_{1},i_{2},\cdots,i_{k}}{\bf 1}_{S_{m+k}^{(i_{1}i_{2}\cdots i_{k})}}~|~\mathcal{F}_{m+k-1}\}~|~\mathcal{F}_{m}\}\label{stepIneq1}\\
&\ge&\frac{\delta}{2}\E\{\sum_{i_{1},i_{2},\cdots,i_{k}}{\bf 1}_{S_{m+k-1}^{(i_{1}i_{2}\cdots i_{k})}}~|~\mathcal{F}_{m}\}\label{stepEq1}\\
&=&\frac{\delta}{2}\E\{\sum_{i_{1},i_{2},\cdots,i_{k-1}}{\bf 1}_{S_{m+k-1}^{(i_{1}i_{2}\cdots i_{k-1})}}~|~\mathcal{F}_{m}\}\label{stepEq2}\\
&\ge&\cdots\nonumber\\
&\ge&(\frac{\delta}{2})^{k-1}\E\{\sum_{i_{1}}{\bf 1}_{S_{m+1}^{(i_{1})}}~|~\mathcal{F}_{m}\}\nonumber\\
&\ge&(\frac{\delta}{2})^{k}\E\{\sum_{i_{1}}{\bf 1}_{S_{m}^{(i_{1})}}~|~\mathcal{F}_{m}\}\nonumber\\
&=&(\frac{\delta}{2})^{k}\nonumber,
\end{eqnarray}
The inequality from (\ref{stepIneq1}) to (\ref{stepEq1}) is due to the fact
\begin{eqnarray*}
\Prb\{S_{m+k}^{(i_{1}i_{2}\cdots i_{k})}~|~\mathcal{F}_{m+k-1}\}\ge \delta/2
\end{eqnarray*}
on $S_{m+k-1}^{(i_{1}i_{2}\cdots i_{k})}\in \mathcal{F}_{m+k-1}$, which implies
\begin{eqnarray*}
\E\{{\bf 1}_{S_{m+k}^{(i_{1}i_{2}\cdots i_{k})}}~|~\mathcal{F}_{m+k-1}\}\ge \frac{\delta}{2} {\bf 1}_{S_{m+k-1}^{(i_{1}i_{2}\cdots i_{k})}}.
\end{eqnarray*}
The equality from (\ref{stepEq1}) to (\ref{stepEq2}) is due to the fact that $S_{m+k-1}^{(i_{1}i_{2}\cdots i_{k-1})}=\cup_{i_{k}}S_{m+k-1}^{(i_{1}i_{2}\cdots i_{k})}$,
and $S_{m+k-1}^{(i_{1}i_{2}\cdots i_{k-1}j)}\cap S_{m+k-1}^{(i_{1}i_{2}\cdots i_{k-1}j')}=\emptyset$ for $j\ne j'$, which implies
\begin{eqnarray*}
{\bf 1}_{S_{m+k-1}^{(i_{1}i_{2}\cdots i_{k-1})}}=\sum_{i_{k}}{\bf 1}_{S_{m+k-1}^{(i_{1}i_{2}\cdots i_{k-1}i_{k})}}
\end{eqnarray*}
for each fixed sequence $i_{1}$, $i_{2}$, $\cdots$, $i_{k-1}$.

Let $h=(n-1)\Sp(n)$, and $\delta'=(\delta/2)^{h}$, then (\ref{eqnProbabilitySpanningTree}) holds. The proof is completed.
\end{proof}

Now we come to the proof of Theorem \ref{thmAppendixMain}.

{\em Proof of Theorem \ref{thmAppendixMain}}

Denote $y_{k}=[x(k), x(k-1),\cdots, x(k-\tau)]^{\top}$, then we have:
\begin{eqnarray}\label{eqnHighDimension}
y_{k+1}=C_{k}y_{k},
\end{eqnarray}
where
\begin{eqnarray*}
C_{k}=\left[\begin{array}{ccccc}
A^{0}(k)& A^{1}(k)& \cdots & A^{\tau-1}(k)& A^{\tau}(k)\\
I&0&\cdots&0& 0\\
0& I & \cdots &0&0\\
\vdots& \vdots & \ddots & \vdots&\vdots\\
0& 0& \cdots & I &0
\end{array}
\right].
\end{eqnarray*}
Let $B_{k'}'=\sum_{k=(k'-1)h+1}^{k'h}B_{k}$, and $\mathcal{F}_{k'}'=\mathcal{F}_{k'h}$. Then $\{B_{k}',\mathcal{F}_{k}'\}$ also forms an adapted sequence, and
$\E\{B_{m+1}'~|~\mathcal{F}_{m}'\}$ has a $\delta$-spanning tree.
From Lemma \ref{lemExpectationToProbability}, there exists $\delta'>0$ such that
\begin{eqnarray*}
\Prb\Big\{\sum_{k=m+1}^{m+(n-1)\Sp(n)}B_{k}'\textnormal{ has a $\delta'$-spanning tree}|~\mathcal{F}_{m}'\Big\}>\delta'.
\end{eqnarray*}

Since
\begin{eqnarray*}
C_{k}&=&\left[\begin{array}{ccccc}
\diag(A^{0}(k))& 0& \cdots & 0& 0\\
I&0&\cdots&0& 0\\
0& I & \cdots &0&0\\
\vdots& \vdots & \ddots & \vdots&\vdots\\
0& 0& \cdots & I &0
\end{array}
\right]
+\left[\begin{array}{ccccc}
\overline{\diag}(A^{0}(k))& A^{1}(k)& \cdots & A^{\tau-1}(k)& A^{\tau}(k)\\
0&0&\cdots&0& 0\\
0& 0 & \cdots &0&0\\
\vdots& \vdots & \ddots & \vdots&\vdots\\
0& 0& \cdots & 0 &0
\end{array}
\right]\\
&\ge&\delta M_{0}+D_{k}\ge \delta(M_{0}+D_{k}),
\end{eqnarray*}
where $D_{k}$ refers to the second matrix on the righthand side of the first line.
From Lemma \ref{lemProdHighDimensionSpanningTree}, Corollary \ref{corollaryProdHighDimensionSpanningTree}, and Remark \ref{remProdHighDimensionSpanningTree}, Remark \ref{remProductSIA}, there exist $\delta''(p)\in (0,\delta')$, $p\ge 1$ such that
\begin{eqnarray*}
\Prb\bigg\{\prod_{k=m+1}^{m+p(n-1)\Sp(n)}C_{k}\textnormal{ is $\delta''(p)$-SIA},p=1,2,\cdots~|~\mathcal{F}_{m}'\bigg\}>\delta'.
\end{eqnarray*}
Let $C'_{k'}=\prod_{k=(k'-1)(n-1)\Sp(n)+1}^{k'(n-1)\Sp(n)}C_{k}$, $\mathcal{F}_{k'}''=\mathcal{F}_{k'(n-1)\Sp(n)}'$, then $\{C_{k}',\mathcal{F}_{k}''\}$ forms an adapted sequence and
\begin{eqnarray*}
\Prb\Big\{\prod_{k=1}^{p}C_{m+k}'\textnormal{ is $\delta''(p)$-SIA },p=1,2,\cdots~|~\mathcal{F}_{m}''\Big\}>\delta'.
\end{eqnarray*}

Let $k_{n}={\bf n}(n)+1$, then from Lemma \ref{lemSIAtoScrambling} and Lemma \ref{lemProbabilityToLongerSeqence},
\begin{eqnarray*}
&&\Prb\Big\{\prod_{i=1}^{k_{n}}C_{m+i}' \textnormal{ is $\delta''(1)^{k_{n}}$-scrambling} |~\mathcal{F}_{m}''\Big\}\\
&\ge&\Prb\Big\{\prod_{k=1}^{p}C_{m'+k}' \textnormal{ is $\delta''(p)$-SIA}, ~m'=m,\cdots,m+k_{n}-1,p=1,2,\cdots.|~\mathcal{F}_{m}''\Big\} \\
&>&\delta'^{k_{n}}.
\end{eqnarray*}
Let $C_{m}''=\prod_{k=(m-1)k_{n}+1}^{mk_{n}}C_{k}'$, and $\mathcal{F}_{m}'''=\mathcal{F}_{mk_{n}}''$.
Then $\{C_{k}'',\mathcal{F}_{k}'''\}$ forms an adapted sequence, and
\begin{eqnarray*}
\Prb\{C_{m+1}'' \textnormal{ is $\delta''(1)^{k_{n}}$-scrambling}~|~\mathcal{F}_{m}'''\}>\delta'^{k_{n}}.
\end{eqnarray*}
From the Second Borel-Cantelli Lemma (Lemma \ref{lemSecondBorelCantelli}), this implies
\begin{eqnarray*}
\Prb\{C_{k}'' \textnormal{ is $\delta''(1)^{k_{n}}$-scrambling infinitely often}\}=1.
\end{eqnarray*}
From Lemma \ref{lemContractionOfProduct} and Remark \ref{remErgodicityOfProducts}, we have
\begin{eqnarray*}
\Prb\Big\{\lim_{k\to+\infty}\Delta\Big(\prod_{m=1}^{k}C_{m}\Big)=0\Big\}=1.
\end{eqnarray*}
This implies the system (\ref{eqnHighDimension}) reaches consensus almost surely. On the other hand it is not difficult to show that consensus of (\ref{eqnHighDimension}) can imply consensus of (\ref{eqnDiscrete}). And the proof is completed.

From Theorem \ref{thmAppendixMain}, we can have the following two corollaries. The first one is for deterministic case.

\begin{corollary}\label{corollaryAppendixDeterministic}
Let $\{A^{0}(k),\cdots,A^{\tau}(k)\}$ be a deterministic sequence with $A^{0}(k)\ge \delta I$ for some $\delta>0$, and $[A^{i}(k)]_{jj}=0$ for $i=1,2,\cdots,\tau$, $j=1,\cdots,n$. If there exists $h>0$ such that $\sum_{k=m+1}^{m+h}B(k)$ has a $\delta$-spanning tree, then the system \eqref{eqnDiscrete} will reach consensus.
\end{corollary}

The second one is for a random sequence that will become independent after a fixed length of time.

\begin{corollary}\label{corollaryAppendixIndependent}
Let $\{A^{0}(k),\cdots,A^{\tau}(k)\}$ be a random sequence with $A^{0}(k)\ge \delta I$ for some $\delta>0$, and $[A^{i}(k)]_{jj}=0$ for $i=1,2,\cdots,\tau$, $j=1,\cdots,n$. And there exists $N>0$ such that $A^{0}(k),\cdots,A^{\tau}(k)$ is independent of $A^{0}(k'),\cdots,A^{\tau}(k')$ whenever $k'\ge k+N$. If there exists $h>0$ such that $\sum_{k=m+1}^{m+h}\E B(k)$ has a $\delta$-spanning tree, then the system \eqref{eqnDiscrete} will reach consensus almost surely.
\end{corollary}
\begin{proof}
Let $\mathcal{F}_{k}=\sigma\{A^{0}(m),\cdots,A^{\tau}(m),~m\le k\}$ be the $\sigma$-algebra formed by $A^{0}(m),\cdots,A^{\tau}(m)$, $m\le k$. Then $\{A^{0}(k),\cdots,A^{\tau}(k),\mathcal{F}_{k}\}$ forms an adapted process. From the assumption on $\{A^{0}(k),\cdots,A^{\tau}(k)\}$, $A^{0}(k'),\cdots,A^{\tau}(k')$ is independent of $\mathcal{F}_{k}$ whenever $k'\ge k+N$. Since $B(k)\ge 0$, we have
$$
\E\{\sum_{m=k+1}^{k+N+h}B(m)|\mathcal{F}_{k}\}\ge \E\{\sum_{m=k+N+1}^{k+N+h}B(m)|\mathcal{F}_{k}\}=\E\{\sum_{m=k+N+1}^{k+N+h}B(m)\}
=\sum_{m=k+N+1}^{k+N+h}\E B(m).
$$
Thus, $\E\{\sum_{m=k+1}^{k+N+h}B(m)|\mathcal{F}_{k}\}$ has a $\delta$-spanning tree and the conclusion follows from Theorem \ref{thmAppendixMain}.
\end{proof}

\end{document}